\documentclass[reqno]{amsart}
\usepackage{amssymb}  
\usepackage{amsmath} 
\usepackage{amsthm}  
\usepackage{stmaryrd} 
\usepackage[pic]{xy}   

\newtheorem{thm}{Theorem}

\newtheorem{lem}[thm]{Lemma}
\newtheorem{cor}[thm]{Corollary}
\newtheorem{prob}{Problem}

\title[Subextensions for co-induced modules]{Subextensions for co-induced modules}
\author{{Andrei V. Zavarnitsine}}%

\address{\textup{\scriptsize
Andrei V. Zavarnitsine\\
Sobolev Institute of Mathematics\\
4, Koptyug av.\\
630090, Novosibirsk, Russia\\
}}
\email{zav@math.nsc.ru}

\date{}

\begin{document}
\begin{abstract}
Using cohomological methods, we prove a criterion for the embedding of a group extension with abelian kernel
into the split extension of a co-induced module. This generalises some earlier similar results.
We also prove an assertion about the conjugacy of complements in split extensions of co-induced modules.
Both results follow from a relation between homomorphisms of certain cohomology groups.
\medskip



\end{abstract}

\maketitle

\section{Introduction}

The natural action of $G=\operatorname{PSL}_n(q)$ on
the projective space $\mathbb{P}^{n-1}$ gives rise to the permutation wreath product of
$L=\mathbb{Z}/r\mathbb{Z}$ and $G$, where $r$ is a prime divisor of $(n,q-1)$. The criterion of when this product contains a subgroup isomorphic
to the nonsplit central extension of $L$ by $G$ was obtained in \cite{16Zav}. Namely, it was proved that the
containments holds iff $r$ does not divide $(q-1)/(n,q-1)$. In the present paper, using some cohomology theory, we generalise this fact by
finding a criterion for embedding extensions with an abelian kernel into a split extension.
To state the results more precisely, we introduce some terminology. In what follows, we use right modules and right composition of maps.

Let $R$ be a commutative ring, $G$ a group (possibly infinite), and let $L$ and $M$ be $RG$-modules. Assume that

\begin{gather}
0\longrightarrow L \stackrel{\varepsilon}{\longrightarrow} M \label{emblm}\\[5pt]
0\longrightarrow L \stackrel{\iota}{\longrightarrow} S\stackrel{\pi}{\longrightarrow} G\longrightarrow 1, \label{exlg} \\[5pt]
0\longrightarrow M \stackrel{\lambda}{\longrightarrow} E \stackrel{\rho}{\longrightarrow} G \longrightarrow 1  \label{exmg}
\end{gather}

\medskip
\noindent are exact sequences of modules and groups, where the conjugation action of $S$ on $L\iota$ agrees with the module
structure of $L$, i.\,e. $(l\iota)^s=l(s\pi)\iota $ for all $l\in L$, $s\in S$, and similarly for $M$ and $E$.
We say that $S$ is a {\em subextension} of $E$ with respect to the embedding $\varepsilon$ if there exists a group homomorphism $\beta$
that makes the following diagram commutative:
\begin{equation}\label{sube}
\begin{array}{c}
\xymatrix@=4.4mm@M=1.5mm{
       &0\ar[d]                  &1\ar@{.>}[d]&&\\
0\ar[r]&L\ar[r]\ar[d]^{\,\varepsilon}&S\ar[r]\ar@{.>}[d]^{\,\beta}&G\ar[r]\ar@{=}[d]&1\\
0\ar[r]&M\ar[r]                  &E\ar[r]                  &G\ar[r]          &1
}\end{array}
\end{equation}

\medskip
\noindent Should $\beta$ exist, it must be a monomorphism, which follows from diagram chase.

The map $\varepsilon$ induces a homomorphism of the second cohomology groups
\begin{equation}\label{h2h}
\varepsilon^{(2)}: H^2(G,L) \longrightarrow H^2(G,M).
\end{equation}
Let $\overline{\delta}\in H^2(G,L)$ and $\overline{\gamma}\in H^2(G,M)$ be the elements that define, respectively, the extensions $S$  and $E$
up to equivalence. The following fact holds.
\begin{lem}\cite[Lemma 2]{13Zav}\label{subc}
In the above notation, $S$ is a subextension of $E$ with respect to $\varepsilon$ if and only if
$\overline{\delta}\varepsilon^{(2)}=\overline{\gamma}$.
\end{lem}
This general criterion sometimes can be made more explicit. For example,
in the situation where $G=\operatorname{PSL}_n(q)$ described earlier, we clearly have a
central extension of $R=\mathbb{Z}/r\mathbb{Z}$
by $G$ as a subextension of the wreath product with respect to the diagonal embedding of the principal $RG$-module into
the permutation module, and the above criterion for the existence of this subextension is purely number-theoretic. Since permutation
modules are co-induced, we can generalise this as follows.

We say that a subgroup $H\leqslant G$ is {\em liftable} to $S$, where $S$ is as in (\ref{exlg}), if $H\pi^{-1}$ splits over $L\iota$.
Given an $RH$-module $N$, we recall that
$$\operatorname{Coind}_H^G(N)=\operatorname{Hom}_{RH}(RG,N)$$
is an $RG$-module with the action of $g\in G$ on $\mu\in \operatorname{Coind}_H^G(N)$ given by
$$(\mu g)(x)=\mu(gx)$$
for all $x\in G$.

Our main result is as follows.

\begin{thm}\label{main} Let $G$ be a group, $H\leqslant G$, and let $L$ be an $RG$-module. Denote $M= \operatorname{Coind}_H^G(L_{H})$ and
let $\varepsilon$ be the canonical embedding
\begin{equation}\label{embe}
0\longrightarrow L \stackrel{\varepsilon}{\longrightarrow} M.
\end{equation}
Then an extension
\begin{equation}\label{exe}
0\longrightarrow L \longrightarrow S \longrightarrow G\longrightarrow 1
\end{equation}
is a subextension of the natural semidirect product
$$
0\longrightarrow M \longrightarrow M\leftthreetimes G \longrightarrow G \longrightarrow 1
$$
with respect to $\varepsilon$ if and only if $H$ is liftable to $S$.
\end{thm}
We recall that the embedding $\varepsilon$ in (\ref{embe}) is the image of the identity map of $L_{H}$ under the
natural isomorphism
$$
\operatorname{Hom}_{RH}(L_{H},L_{H})\cong \operatorname{Hom}_{RG}(L,\operatorname{Coind}_H^G(L_{H})).
$$
Explicitly, we have
\begin{equation}\label{le1}
  (l\varepsilon)(g)=lg
\end{equation}
for all $l\in L$, $g\in G$, see \cite[Corollary 2.8.3(ii)]{98Ben}.

A few remarks are due about Theorem \ref{main}.
Suppose a group $S$ has an abelian normal subgroup $L$ and quotient $G=S/L$. Then conjugation defines on $L$ the structure of a $\mathbb{Z}G$-module.
If we take $H$ to be
the trivial subgroup of $G$ then Theorem \ref{main} ensures existence of the embedding $S\to M\leftthreetimes G$, where $M=\operatorname{Coind}_H^G(L_{H})$.
It is readily seen that in this case $M\leftthreetimes G$ is isomorphic to the
unrestricted regular wreath product $L\operatorname{wr} G$ and hence the embedding $S\to M\leftthreetimes G$ also follows from
\begin{thm}[Kaloujnine--Krasner, \cite{48KalKra}]\label{KK} Every group $S$ with a normal subgroup $L$ can be embedded into the unrestricted regular
wreath product $L\operatorname{wr} S/L$.
\end{thm}
\noindent
Therefore, we give and alternative cohomological proof of this result in the case of abelian $L$ and specify a necessary and sufficient condition
for the embedding.

Now, let $L$ be the principal $RG$-module and suppose that the index $|G:H|$ is finite. Then $M$ is just the transitive permutation
module corresponding to the action of $G$ on the cosets of $H$ and $L\varepsilon$ is its diagonal submodule. In \cite{17Zav}, we have considered this situation
restricted to the case where $R$ has prime characteristic but generalised to not necessarily transitive action and shown without using cohomology that
the liftability of $H$ to $S$ is necessary for the existence of the required subextension which must be a central extension in this case.
Conversely, the sufficiency of liftability in the general case can also be deduced
without applying cohomological methods using a generalisation of the Kaloujnine--Krasner theorem \cite[Theorem 2.10.9]{82Suz} which
is originally due to B.\,H.\,Neumann and is related to the so-called twisted wreath products.

As we show below, Theorem \ref{main} follows from a group-theoretic interpretation in dimension 2 of the equality of kernels
of homomorphisms between certain cohomology groups (see Corollary \ref{ker}) which holds in arbitrary dimension. Since cohomology
in dimension 1 is usually also meaningful for groups, we prove the corresponding corollary as well which is as follows.

\begin{thm}\label{aux} Let $G$ be a group, $H\leqslant G$, and let $L$ be an $RG$-module. Denote $M=\operatorname{Coind}_H^G(L_{H})$ and
let $\varepsilon$ be the canonical embedding $(\ref{embe})$.
Then a complement to $L$ in $L\leftthreetimes G$ is $M$-conjugate to $G$ if and only if its intersection with $L\leftthreetimes H$ is $L$-conjugate to $H$.
\end{thm}
\noindent
In the statement of Theorem \ref{aux}, we assume that $L\leftthreetimes G$ is embedded in $M \leftthreetimes G$ via $(g,l)\mapsto (g,l\varepsilon)$ for $g\in G$, $l\in L$, and by $X$-conjugacy we mean the conjugacy by elements of $X$, where $X\in \{M,L\}$.

\section{$H^n$ as a functor}

We recall that $H^n$, $n\geqslant 0$, can be viewed as a functor from the category of pairs $(G,M)$, where $M$ is a $G$-module, see \cite[\S III.8]{82Bro}.
A morphism in this category is a map
$$
(\alpha,\varphi):(H,N)\to (G,M)
$$
with $\alpha:H\to G$ a group homomorphism and $\varphi:M\to N$ a homomorphism of $H$-modules, where $M$ is considered as an $H$-module via $\alpha$, i.\,e.
\begin{equation}\label{comcon}
(m(h\alpha))\varphi=(m\varphi)h
\end{equation}
for all $m\in M$, $h\in H$. It gives rise to a homomorphism
$$(\alpha,\varphi)^{(n)}:H^n(G,M)\to H^n(H,N).$$
By considering the standard (normalised) projective resolutions for $N$ and $M$,
it can be seen that $(\alpha,\varphi)^{(n)}$ is induced from the chain map $C^n(G,M)\to C^n(H,N)$ on (normalised) cochains which we also denote by
$(\alpha,\varphi)^{(n)}$ and which is given by
$$
\lambda(\alpha,\varphi)^{(n)}= (\alpha\times \ldots \times\alpha)\lambda \varphi
$$
for every $\lambda\in C^n(G,M)$. Three particular cases are of interest to us.

$(i)$ Suppose that $H=G$ and $\alpha=\operatorname{id}_H$. Then we denote $\varphi^{(n)}=(\alpha,\varphi)^{(n)}$ which is just
the standard induced homomorphism $H^n(G,\varphi)$ in this case. In particular, $\lambda \varphi^{(n)}=\lambda \varphi$ for $\lambda\in C^n(G,M)$.

$(ii)$ Suppose that $\alpha:H\hookrightarrow G$ is an embedding and $N=M_H$. If $\varphi=\operatorname{id}_{M}$ then the compatibility
condition (\ref{comcon}) holds and we denote $\alpha^{(n)}=(\alpha,\varphi)^{(n)}$. In particular, $\lambda \alpha^{(n)} = (\alpha\times\ldots \times\alpha)\lambda$
for $\lambda\in C^n(G,M)$.

$(iii)$ Suppose that $\alpha:H\hookrightarrow G$ is an embedding and $M=\operatorname{Coind}^G_H(N)$. If $\varphi:M\to N$ is the canonical
epimorphism
\begin{equation}\label{fact}
\mu \varphi = \mu(1),
\end{equation}
where $\mu \in M$, the compatibility condition (\ref{comcon}) holds.
In this case, the induced map $(\alpha,\varphi)^{(n)}:H^n(G,M)\to H^n(H,N)$ is known to be an isomorphism due to the following result.

\begin{lem}[Shapiro's lemma, {\cite[\S 6.3]{94Wei}}] \label{shap}
If $H\leqslant G$ and $N$ is an $H$-module then $H^n(G,\operatorname{Coind}^G_H(N))\cong H^n(H,N)$.
\end{lem}

\noindent The fact that the isomorphism in Shaprio's lemma coincides with the map $(\alpha,\varphi)^{(n)}$ is well known,
see \cite[Proposition (III.6.2) and \S 8, Exercise 2]{82Bro}.

\section{Co-induced modules}

Let $\alpha: H\hookrightarrow  G$ be an embedding of groups and let $L$ be a $G$-module. Denote $M= \operatorname{Coind}_H^G(L_H)$.
The canonical embedding $\varepsilon: L \to M$ gives rise to a homomorphism $\varepsilon^{(n)}:H^n(G,L)\to H^n(G,M)$ as in $(i)$ above.
By the previous discussion, we also have the homomorphisms $\alpha^{(n)}$ and $(\alpha,\varphi)^{(n)}$ which fit into the diagram
\begin{equation}\label{diag}
\xymatrix{
H^n(G,L)\ar[r]^{ \varepsilon^{(n)}} \ar[d]_{ \alpha^{(n)}} & \ar[dl]^{\ (\alpha,\varphi)^{(n)} } H^n(G,M) \\
H^n(H,L_H) &
}
\end{equation}
where the map $\varphi:M\to L_H$ is as in (\ref{fact}).

\begin{lem}\label{comd}
Diagram $(\ref{diag})$ is commutative.
\end{lem}
\begin{proof}
It suffices to check that $\lambda \varepsilon^{(n)} (\alpha,\varphi)^{(n)} = \lambda \alpha^{(n)}$
for every $\lambda \in C^n(G,L)$. By $(i)$--$(iii)$ above, we have
$$
(\lambda \varepsilon^{(n)} )(\alpha,\varphi)^{(n)} = (\lambda \varepsilon )(\alpha,\varphi)^{(n)} = (\alpha\times \ldots \times\alpha)\lambda \varepsilon \varphi =
(\alpha\times \ldots \times\alpha) \lambda = \lambda \alpha^{(n)},
$$
since $\varepsilon \varphi = \operatorname{id}_L$ due to (\ref{le1}) and (\ref{fact}). The claim follows.
\end{proof}
The map $(\alpha,\varphi)^{(n)}$ is an isomorphism by Lemma \ref{shap}. Therefore, Lemma \ref{comd} implies

\begin{cor}\label{ker}
$\operatorname{Ker} \varepsilon^{(n)} = \operatorname{Ker} \alpha^{(n)}$.
\end{cor}

We note that henceforth instead of $G$-modules we may as well consider arbitrary $RG$-modules. This
follows from the next result which essentially says that co-induced modules and cohomology groups are independent of the ground ring.

\begin{lem}\label{extr} For $H\leqslant G$, let $M$ be an $RG$-module and $N$ an $RH$-module. Then the following isomorphisms of
abelian groups hold:
\begin{enumerate}
  \item[$(i)$] $\operatorname{Hom}_{RH}(RG,N)\cong\operatorname{Hom}_{\mathbb{Z}H}(\mathbb{Z}G,N)$;
  \item[$(ii)$] $\operatorname{Ext}^n_{RG}(R,M)\cong \operatorname{Ext}^n_{\mathbb{Z}G}(\mathbb{Z},M)$.
\end{enumerate}
\end{lem}
\begin{proof}
$(i)$ Both abelian groups equal
$$
\{f:G\to N\mid (gh)f=(gf)h \quad \forall g\in G, h\in H\}
$$
with the natural additive structure.

$(ii)$ See \cite[Lemma 9.4.13]{02LeeMcK}.
\end{proof}

\section{Proof of main results}

We now prove Theorem \ref{main}.

\begin{proof} Since the split extension $M\leftthreetimes G$ is defined by the zero element of $H^2(G,M)$, Lemma \ref{h2h}
implies that $S$ is a subextension of $M\leftthreetimes G$ with respect to $\varepsilon$ if and only if $\overline{\delta}\in \operatorname{Ker} \varepsilon^{(2)}$,
where $\overline{\delta}\in H^2(G,L)$ defines $S$. By Corollary \ref{ker} specialised to dimension $2$, we have $\operatorname{Ker} \varepsilon^{(2)}=\operatorname{Ker} \alpha^{(2)}$,
where $\alpha^{(2)}: H^2(G,L)\to H^2(H,L_H)$ and $\alpha$ is the embedding $H\hookrightarrow G$.
However, $\overline{\delta}$ lies in $\operatorname{Ker} \alpha^{(2)}$ if and only if it is
mapped to the zero element of $H^2(H,L_H)$ which defines the split extension $L\rightthreetimes H$, i.\,e. this is possible if and only if
$H$ is liftable to $S$, as is required. \end{proof}

In a similar fashion, Theorem \ref{aux} can be proved as follows.

\begin{proof} The $L$-conjugacy classes of complements to $L$ in $L\leftthreetimes G$ are in a one-to-one correspondence with the elements of $H^1(G,L)$ with the
class of $G$ corresponding to the zero of $H^1(G,L)$, see \cite[11.1.3]{96Rob}. Therefore, by considering the
action on $1$-cocycles, one sees that the elements
of the kernel of $\varepsilon^{(1)}:H^1(G,L)\to H^1(G,M)$ correspond to the $L$-conjugacy classes of complements
in $L\leftthreetimes G$ that merge to the $M$-conjugacy class of $G$.
On the other hand, Corollary \ref{ker} specialised to dimension $1$ implies that
$\operatorname{Ker} \varepsilon^{(1)}=\operatorname{Ker} \alpha^{(1)}$. Again, by considering the action on $1$-cocycles, we see that the elements of the kernel of $\alpha^{(1)}: H^1(G,L)\to H^1(H,L_H)$ correspond to  the $L$-conjugacy classes of complements in $L\leftthreetimes G$ that intersect $L\leftthreetimes H$ in an $L$-conjugate of $H$. The claim follows from these remarks.
\end{proof}

\section{Defining subgroups}

Given an $RG$-module $L$ and a subgroup $H\leqslant G$, we say that an extension
\begin{equation}\label{lse}
0\longrightarrow L \stackrel{\iota}{\longrightarrow} S\stackrel{\pi}{\longrightarrow} G\longrightarrow 1
\end{equation}
is {\em defined} by $H$ if $L$ is a subextension of $M\leftthreetimes G$, where $M=\operatorname{Coind}_H^G(L_{H})$,
with respect to the natural embedding $\varepsilon: L\to M$ given in (\ref{le1}).

\begin{lem}\label{defs} Let $H\leqslant G$, let $L$ be an $RG$-module, and let $S$ be the extension $(\ref{lse})$ that is defined by $H$. Then
\begin{enumerate}
\item[$(i)$] $S$ is defined by $K$ for every $K\leqslant H$;
\item[$(ii)$] $S$ is defined by $H^g$ for every $g\in G$.
\end{enumerate}
\end{lem}
\begin{proof} By Theorem \ref{main}, the fact that $S$ is defined by $H$ is equivalent to the liftability of $H$ to $S$ which clearly implies
the liftability of both $K$ and $H^g$, hence the claim.

Observe that we can also prove this lemma without using Theorem \ref{main}. Indeed,
let $M=\operatorname{Coind}_H^G(L_{H})$ and let $\beta:S\to M\leftthreetimes G$ be the subextension embedding.

First, suppose $K\leqslant H$ and denote $N=\operatorname{Coind}_K^G(L_{K})$. There is a canonical $RG$-embedding $\varphi: M\to N$
which acts identically on every element of $M$ viewed as a map $G\to L$. In particular, $\delta=\varepsilon\varphi$ is the natural embedding $L\to N$.
Also, $\varphi$ uniquely extends to a map $\alpha: M\leftthreetimes G\to N\leftthreetimes G$ so that $\beta\alpha$ gives the required
subextension embedding $S\to N\leftthreetimes G$ with respect to $\delta$.
$$
\begin{array}{c}
\xymatrix@=4.4mm@M=1.5mm{
0\ar[r]&L\ar[r]\ar[d]^{\,\varepsilon}&S\ar[r]\ar[d]^{\,\beta}&G\ar[r]\ar@{=}[d]&1\\
0\ar[r]&M\ar[r]\ar[d]^{\,\varphi}&M\leftthreetimes G\ar[r]\ar[d]^{\,\alpha}&G\ar[r]\ar@{=}[d]&1\\
0\ar[r]&N\ar[r]                  &N\leftthreetimes G\ar[r]                  &G\ar[r]          &1
}\end{array}
$$

Second, suppose $g\in G$ and denote $U=\operatorname{Coind}_{H^g}^G(L_{H^g})$.
Since the $RH$- and $RH^g$-modules $L_{H}$ and $L_{H^g}$ are conjugate by $g$,
there is an $RG$-isomorphism $\psi: M\to U$ given by $(\mu\psi)(x)=\mu(xg^{-1})g$ for all $\mu \in M$, $x\in G$.
We see that $\varepsilon\psi$ is the natural embedding $L\to U$, because $(l\varepsilon\psi)(x)=(l\varepsilon)(xg^{-1})g=lxg^{-1}g=lx$.
Hence, as above we have the required subextension embedding $S\to U\leftthreetimes G$.
\end{proof}

By Lemma \ref{defs}, the study of defining subgroups for a given extension (\ref{lse}) reduces to the the study up to conjugacy
of maximal liftable to $S$ subgroups of $G$. The set of such subgroups is nonempty as the identity subgroup
is always liftable.

For example, consider the particular case of alternating groups and their central double covers.

\begin{prob}\label{2an} Let $G=A_n$ be the alternating group of degree $n\geqslant 4$ and let $S=2.A_n$ be its nonsplit double cover.
\begin{enumerate}
\item[$(i)$] Describe the maximal liftable to $S$ subgroups of $G$.
\item[$(ii)$] Find a function $f:\mathbb{N}\to\mathbb{N}$ such that $f(n)$ is the minimal number with the property that $S$ is embedded
to a semidirect product $M\leftthreetimes G$, where $M$ is an elementary abelian group of order $2^{f(n)}$.
\item[$(iii)$] Describe the maximal subgroups of $G$ that lift to $S$.
\end{enumerate}
\end{prob}

\noindent
It follows from Theorem \ref{main} that the value $f(n)$ in item $(ii)$ is bounded above by the minimal index of liftable subgroups.
The case $(iii)$, where a maximal subgroup of $G$ is liftable to $S$, is of special interest, because we then obtain the most `economic'
subextension embedding in view of Lemma \ref{defs}$(i)$. This need not always happen, however, as we saw, for example,
in the case $G=\operatorname{PSL}_n(q)$ above. For $G=A_n$, it can be shown that no maximal subgroup is liftable to $2.A_n$ for $n=5,6,7,8$,
but there are three conjugacy classes of maximal subgroup of $A_9$ that lift to $2.A_9$. These subgroups
have indices $120$ (two classes) and~$840$ (one class) and are isomorphic to $\operatorname{L}_2(8)\!:\!3$ and $\operatorname{ASL}_2(3)$, respectively.

\providecommand{\bysame}{\leavevmode\hbox to3em{\hrulefill}\thinspace}
\providecommand{\MR}{\relax\ifhmode\unskip\space\fi MR }
\providecommand{\MRhref}[2]{%
  \href{http://www.ams.org/mathscinet-getitem?mr=#1}{#2}
}
\providecommand{\href}[2]{#2}

\end{document}